\documentclass[12pt,final]{article}
\usepackage{amsmath,amsthm,amsfonts,amssymb,graphicx,enumerate,psfrag}
\usepackage{tikz,pgfplots}
\usepackage{mathrsfs}
\usepackage{fullpage}
\usepackage{mathtools}

\usepackage[colorlinks,citecolor=blue,urlcolor=blue]{hyperref}
\definecolor{myblue}{RGB}{51,51,178}
\definecolor{myred}{RGB}{189,26,26}
\definecolor{mygreen}{RGB}{0,128,0}
\usepackage[utf8]{inputenc} 
\newtheorem{lemma}{Lemma}

\newtheorem{theorem}{Theorem}
\newtheorem{example}{Example}
\newtheorem{proposition}{Proposition}
\newtheorem{corollary}{Corollary}
\newtheorem{remark}{Remark}
\newtheorem{conjecture}{Conjecture}

\newcommand{\dN}{\mathbb {N}}

\newcommand{\cP}{\mathcal {P}}

\newcommand{\cA}{\mathcal {A}}
\newcommand{\cE}{\mathcal {E}}

\newcommand{\dR}{\mathbb {R}}
\newcommand{\CD}{\mathrm{CD}}
\newcommand{\dd}{\mathrm{d}}

\newcommand{\EE}{{\mathbb{E}}}

\newcommand{\dX}{{\mathbb{X}}}

\newcommand{\tmix}{\mathrm{t}_{\textsc{mix}}}
\newcommand{\wmix}{\mathrm{w}_{\textsc{mix}}}

\newcommand{\dtv}{\mathrm{d}_{\textsc{tv}}}

\newcommand{\Var}{{\mathrm{Var}}}

\newcommand{\supp}{\mathrm{supp}}

\DeclarePairedDelimiter{\abs}{\lvert}{\rvert}
\DeclarePairedDelimiter{\tond}{(}{)}

\DeclarePairedDelimiter{\scal}{\langle}{\rangle} 
\newcommand{\chisq}[2]{\chi^2\tond*{#1\,\middle\|\,  #2}}
\newcommand{\expo}[1]{\exp\tond*{#1}}

\newcommand{\relden}[2]{\frac{\dd#1}{\dd#2}}

\newcommand{\variwr}[2]{\mathrm{Var}_{#1}\tond*{#2}}

\title{The local product condition implies cutoff}
\author{Francesco Pedrotti and Justin Salez}
\begin{document}
	\maketitle
	\begin{abstract}
		In the theory of mixing times, a famously wrong conjecture predicts that a sequence of Markov processes exhibits cutoff as soon as the product of their Poincaré constant and mixing time diverges. We prove that this statement becomes correct once the Poincaré constant $\gamma$ is replaced with its natural non-equilibrium refinement, which we denote by $\gamma_\star$. More precisely, we show that the width of the mixing window of any Markov process is $O(1/\gamma_\star)$.  This  estimate is sharp, and universal up to standard regularity assumptions: it holds on finite and infinite state spaces  and from any initial condition, and it does not  require reversibility, nor any kind of a chain rule. In addition, for deterministic initialization we show that $\gamma_\star\ge\kappa$, where $\kappa$ is the Bakry-\'Emery curvature, making our result broadly applicable. Finally, our proof is short and self-contained: we simply follow the classical idea of replacing the total variation distance  by the more tractable $\chi^2$-divergence, but with the  crucial novelty that the reference measure  evolves in time, instead of being the equilibrium law. 
	\end{abstract}

	\section{Introduction}
	Throughout this paper, we consider a continuous-time Markov process $(X_t)_{t\ge 0}$ taking values in a measurable space $(\dX,\cE)$, with generator $L$, semigroup $(P_t)_{t\ge 0}$ and invariant law $\pi\in\cP(\dX)$. As usual, we  assume that  $\mu_t:=\mathrm{law}(X_t)$ is absolutely continuous w.r.t. $\pi$ for each $t>0$, and that   $\dtv(\mu_t,\pi)\to 0$ as $t\to\infty$, where $\dtv$ denotes the  total variation distance:
	\begin{eqnarray*}
		\dtv(\mu,\pi) & := & \sup_{A\in\cE}\,\abs*{\mu(A)-\pi(A)}.
	\end{eqnarray*} We also assume to be given an algebra $\cA \subset L^\infty(\pi)$ 
	which contains constants, is contained in the domain of $L$, is stable under  $(P_t)_{t\ge 0}$, and is dense in $L^2(\pi)$. These assumptions are trivially satisfied when $\dX$ is finite, but they hold considerably more generally \cite{MR3155209}. Given a precision $\varepsilon \in (0,1)$, we define the mixing time as
	\begin{eqnarray*}
		\tmix(\varepsilon) & := & \inf\left\{t\ge 0\colon \dtv\tond*{\mu_t, \pi}\le \varepsilon\right\}.
	\end{eqnarray*}
	
	In many concrete examples, the process $(X_t)_{t\ge 0}$ of interest actually depends on a \emph{dimension parameter} $n\in\dN$ which affects the state space, the initial condition and the dynamics, and one is interested in the \emph{high-dimensional} limit where $n\to\infty$. In this regime, several models are known to undergo a sharp transition from out-of-equilibrium to equilibrium known as a \emph{cutoff}, and defined as follows (see Figure \ref{fig:cutoff}):
	\begin{eqnarray}
		\label{def:cutoff}
		\forall\varepsilon\in(0,1),\qquad \frac{\tmix(1-\varepsilon)}{\tmix(\varepsilon)} & \xrightarrow[n\to\infty]{} & 1.
	\end{eqnarray}
	\begin{figure}
		\begin{center}
			\pgfplotsset{compat=1.16}
			\pgfplotsset{ticks=none}
			\begin{tikzpicture}
				\begin{axis}[
					width = 14.5cm,
					height = 9cm,
					axis x line=middle,
					axis y line=middle,
					xlabel = $t$,
					xlabel style ={at={(1,-0.1)}},
					ylabel style ={at={(0,1.05)}},
					ylabel = $\dtv\tond{\mu _t, \pi}$,
					clip = false,
					grid=both,
					grid style={dashed, line width=.5pt, draw=gray!10},
					xmode = normal,
					ymode = normal,
					line width = 1pt,
					legend cell align = left,
					legend style = {fill=none, at={(0.9,0.9)}, anchor = north east},
					yticklabel style={above left},
					anchor = north west,
					tickwidth={5pt},
					xtick align = outside,
					ytick align = outside,
					ymax = 1.05,
					ymin = 0,
					xmin = 0,
					xmax = 40,
					x axis line style=-,
					y axis line style=-,
					domain = 0:40,
					samples = 1000
					]
					\def\scale{3}
					\addplot[solid,  myred, line width = 1.5 pt] {
						-rad(atan(x-20))/rad(atan(20))/2 + 0.5	
					};
					%
					\draw[dashed, color= myblue] (15,0) -- (15,0.9515);
					\draw[dashed, color= myblue] (0,0.9515) -- (15,0.9515);
					\draw[dashed, color= mygreen] (25,0) -- (25,0.0484723);
					\draw[dashed, color= mygreen] (0,0.0484723) -- (25,0.0484723);	

					\draw[color= myblue] (0,0.9515) -- (-0.5,0.9515) node[color = black, anchor=east] {$\textcolor{myblue}{1-\varepsilon}$};
					\draw[color= myblue] (15,0) -- (15,-0.02) node[color=black, anchor=north] {$\textcolor{myblue}{\tmix(1-\varepsilon)}$};
					\draw[color= mygreen] (25,0) -- (25,-0.02) node[color = black, anchor=north] {${\textcolor{mygreen}{\tmix(\varepsilon)}}$};
					\draw[color= mygreen] (0,0.0484723) -- (-0.5,0.0484723) node[color = black,anchor=east] {$\textcolor{mygreen}{\varepsilon}$};
					
					\draw[color= black] (0,1.0) -- (-0.5,1.0) node[color = black,anchor=south east] {${1}$};
					
				\end{axis}
				
			\end{tikzpicture}
			\caption{A typical plot of the distance to equilibrium $t\mapsto \dtv\tond{\mu_t, \pi}$. As the ratio $\frac{\tmix(1-\varepsilon)}{\tmix(\varepsilon)}$ approaches $1$, the transition to equilibrium becomes abrupt (cutoff).}
			\label{fig:cutoff}
		\end{center}
	\end{figure}
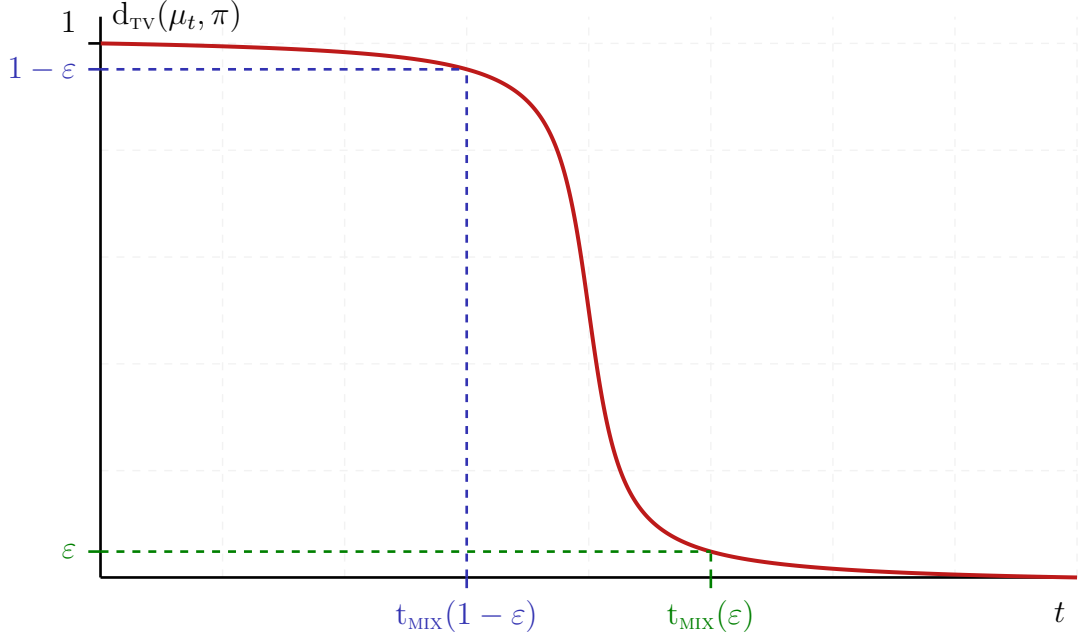
	We refer the unfamiliar reader to the seminal works \cite{diaconis1981generating,aldous1983mixing,aldous1986shuffling,diaconis1996cutoff}, the recent paper \cite{MR4780485} and the Saint-Flour lecture notes \cite{salez2025lecture} for an introduction to this beautiful phenomenon. Despite the accumulation of hundreds of examples, a unified theory is still missing and identifying the general conditions that trigger a cutoff constitutes one of the most challenging open problems in the modern analysis of Markov processes. In particular, the following famously wrong, but extremely fertile answer was proposed by Peres at an AIM Workshop ``\emph{Sharp Thresholds for Mixing Times}'' in 2004.  
	\begin{conjecture}[The product condition implies cutoff]\label{conj:main} A cutoff occurs as soon as 
		\begin{eqnarray}
			\label{def:PC}
			\gamma\times \tmix(\varepsilon) & \xrightarrow[n\to\infty]{} & +\infty,
		\end{eqnarray}
		for some $\varepsilon\in(0,1)$, where $\gamma$ denotes the Poincaré constant of the considered process.  
	\end{conjecture}
	Let us recall that the Poincaré constant is the largest number $\gamma\ge 0$ such that
	\begin{eqnarray}
		\label{def:PI}
		\forall f\in \cA,\qquad \gamma\,\Var_\pi(f) & \le & \int \Gamma f\,\dd\pi,
	\end{eqnarray}
	where $\Gamma\colon \cA\to L^2(\pi)$ is the \emph{carré du champ} operator given by
	\begin{eqnarray*}
		\label{eq:Gamma-def}
		\Gamma f & := & \frac12 L f^2 - f Lf.
	\end{eqnarray*}
	Conjecture \ref{conj:main} has drawn considerable attention in the community of mixing times, mainly because of its effective nature: unlike   (\ref{def:cutoff}), the product condition (\ref{def:PC}) is easily checked in practice, since  it only requires rough lower bounds on $\gamma$ and $\tmix$. Moreover, it is known to be necessary for cutoff, at least in the standard setup of reversible chains on finite state spaces started from worst-case initial conditions (see, e.g., \cite{salez2025lecture}). It has also been shown to be sufficient on particular ensembles, including birth-and-death processes \cite{ding2010total}, random walks on trees \cite{MR3650406}, non-conservative exclusion dynamics \cite{MR4546624}, and non-negatively curved diffusions \cite{salez2025diffusions}. Unfortunately,  it is \emph{not} sufficient in general \cite{MR2774096}, and counter-examples are in a sense \emph{generic}. More precisely, a simple additive perturbation of the form
	\begin{eqnarray*}
		\widetilde{L}f & := & Lf+\eta\left(\int f\dd\pi-f\right),
	\end{eqnarray*}
	will, if the perturbation parameter $\eta$ is tuned appropriately, completely destroy cutoff while preserving most of the structural properties of the process, including the product condition (see again \cite{salez2025lecture}). Nevertheless, counter-examples of this form are widely considered as \emph{pathological} by the community, and the general consensus is that the product condition should correctly predict cutoff for \emph{reasonable} processes. Giving an honest mathematical meaning to this vague claim is precisely the aim of the present paper. 
	
	\section{Main result}

	The Poincaré inequality (\ref{def:PI}) admits a natural generalization, obtained by replacing the invariant law $\pi$ with any measure $\mu\in\cP(\dX)$ which is absolutely continuous w.r.t. $\pi$. Specifically, we define the Poincaré constant  of  $\mu$ as the largest constant $\gamma(\mu)\in[0,\infty]$ such that
	\begin{eqnarray}
		\label{def:PIgeneral}
		\forall f\in\cA,\qquad \gamma(\mu)\,\Var_\mu\left(f\right) & \le & \int \Gamma f\,\dd\mu.
	\end{eqnarray}
	We  then define the \emph{local Poincaré constant} of our Markov process $(X_t)_{t\ge 0}$ as follows:
	\begin{eqnarray*}
		\gamma_\star & := & \inf_{t> 0}\gamma(\mu_t).
	\end{eqnarray*}
	Equivalently, $\gamma_\star$ is the optimal constant in the \emph{local Poincaré inequality}
	\begin{eqnarray}
		\label{def:LPI}
		\forall f\in\cA,\quad \forall t> 0,\quad \gamma_\star \Var\left(f(X_t)\right) & \le & \EE\left[\Gamma f(X_t)\right].
	\end{eqnarray}
	Note that this definition strengthens the classical one, in the sense that $\gamma_\star\le \gamma$, as can be seen by sending $t\to\infty$ in the above inequality (recall that $\dtv(\mu_t,\pi)\to 0$ as $t\to\infty$). Our main result asserts that this natural non-equilibrium refinement suffices to repair Conjecture \ref{conj:main}. More quantitatively, we prove that the inverse local Poincaré constant $\gamma_\star^{-1}$ is a universal upper bound on  the \emph{width} of the mixing window, defined as follows (with $0<\varepsilon<1/2$):
	\begin{eqnarray*}
		\wmix(\varepsilon) & := & \tmix(\varepsilon)-\tmix(1-\varepsilon).
	\end{eqnarray*}
	We emphasize that our estimate holds on arbitrary state spaces and from any initial condition. Neither reversibility, nor any kind of chain rule are required.
	\begin{theorem}[The local product condition implies cutoff]\label{th:main}For any $0<\varepsilon\le 1/2$, we have
		\begin{eqnarray}
			\label{eq:window}
			\wmix(\varepsilon) & \le & \frac{3}{2\gamma_\star}\log \frac{1}{\varepsilon}.
		\end{eqnarray}
		In particular, the cutoff phenomenon occurs as soon as, for some  (hence every) $\varepsilon\in(0,1)$,
		\begin{eqnarray}
			\label{def:LPC}
			\gamma_\star \times \tmix(\varepsilon) & \xrightarrow[n\to\infty]{} & \infty.
		\end{eqnarray}
	\end{theorem} Recall that Conjecture \ref{conj:main} is widely believed to correctly predict cutoff for all  \emph{reasonable} processes, in a sense that has long resisted any rigorous formalization. In light of Theorem \ref{th:main}, we see that any model along which the ratio $\gamma/\gamma_\star$ is bounded independently of the dimension parameter $n$ may safely be deemed \emph{reasonable}. An emblematic example is the following fundamental toy model, which also shows that our window estimate (\ref{eq:window}) is sharp. 
	\begin{example}[Hypercube]\label{ex:cube}Let $(X_t)_{t\ge 0}$ be the  random walk on  $\dX=\{0,1\}^n$  with  generator 
		\begin{eqnarray*}
			Lf(x) & := & \sum_{i=1}^n\left(f(x^i)-f(x)\right),
		\end{eqnarray*}
		where $x\mapsto x^i$ is the map that flips the $i$-th coordinate. The associated carré du champ  is 
		\begin{eqnarray*}
			\Gamma f(x) & := & \frac{1}{2}\sum_{i=1}^n\left(f(x^i)-f(x)\right)^2.
		\end{eqnarray*}
		Now, if $X_0$ has a product law (which is the case when $X_0$ is non-random), then so does  $X_t$ for any time $t\ge 0$. Consequently,  the Efron-Stein inequality \cite{MR3185193} guarantees that 
		\begin{eqnarray*}
			\Var\left(f(X_t)\right) & \le &  \frac{1}{2}\EE\left[\Gamma f(X_t)\right],
		\end{eqnarray*}
		for all $f\in\dR^\dX$. In other words, we have $\gamma_\star\ge 2$, and since $\gamma_\star\le \gamma=2$, we conclude that
		\begin{eqnarray*}
			\gamma_\star & = & \gamma \ = \ 2.
		\end{eqnarray*}
		Thus, in this case, the local product condition (\ref{def:LPC}) is exactly  the classical product condition, and we immediately deduce cutoff with a window $\wmix(\varepsilon)=O(1)$, which is  sharp. 
	\end{example}
	To become effective, Theorem \ref{th:main} needs to be complemented with a systematic method for estimating the local Poincaré constant $\gamma_\star$. To this end, we exploit an idea that lies at the heart of the Bakry-\'Emery theory \cite{MR889476,MR3155209}, namely, 
	a sub-commutation relation between the carré du champ and the  semigroup.  Specifically, for each $t\ge 0$, we let $\psi(t)\in[0,\infty]$ denote the optimal constant in the following comparison of quadratic forms:
	\begin{eqnarray}
		\label{def:psi}
		\forall f\in\cA,\qquad \Gamma P_t f & \le & \psi(t)\,P_t \Gamma f.
	\end{eqnarray}
	The interest of this definition is summarized in the following result.

	\begin{proposition}[Controlling the local Poincaré constant]\label{pr:psi}
		If $\mu_0 \ll \pi$, then
		\begin{eqnarray}\label{eq:local-pi-estimate}
			\forall t>0,\quad \frac{1}{\gamma(\mu_t)} & \le & \frac{\psi(t)}{\gamma(\mu_0)}+2\int_0^t\psi(s)\dd s.
		\end{eqnarray} 
		If instead of $\mu_0\ll\pi$ we assume that 
		\begin{eqnarray}\label{eq:ass-det-init}
			\forall f \in \cA, \quad \variwr{\mu_t}{f} & \xrightarrow[t\to 0^+]{} & 0. 
		\end{eqnarray}
		then 
		\begin{eqnarray}\label{eq:local-pi-estimate-deterministic}
			\forall t>0 ,\quad \frac{1}{\gamma(\mu_t)} & \le & 2\int_0^t\psi(s)\dd s.
		\end{eqnarray} 
		
	\end{proposition}
	The assumption \eqref{eq:ass-det-init} morally corresponds to the situation where $\mu_0$ is a Dirac mass,  in which case one can set  $\gamma(\mu_0) := \infty$, so that  \eqref{eq:local-pi-estimate-deterministic} agrees with \eqref{eq:local-pi-estimate}. Note, however, that in our general setup,  functions in $\cA$ are a priori only defined up to a $\pi$-negligible set, while Dirac masses are typically not absolutely continuous w.r.t.  $\pi$, hence the  abstract condition \eqref{eq:ass-det-init}.
	
	By the semigroup property, we automatically have $\psi(0)=1$ and $\psi(t+s) \le \psi(t)\psi(s)$ for all $t,s\ge 0$. Consequently, an exponential estimate of the form 
	\begin{eqnarray}
		\label{def:kappa}
		\psi(t) & \le & e^{-2\kappa t},
	\end{eqnarray}
	will hold for all $t\ge 0$ as soon as it holds in an infinitesimal neighborhood of $t=0$, and an easy differentiation leads to the celebrated $\CD(\kappa,\infty)$ condition introduced in \cite{MR889476}:
	\begin{eqnarray}
		\label{CD}
		\forall f\in\cA,\qquad \Gamma_2 f & \ge &  \kappa\, \Gamma f,
	\end{eqnarray}
	where $\Gamma_2 f:=\frac{1}{2}L\Gamma f-\Gamma(f,Lf)$ is the iterated carré du champ operator. The optimal constant $\kappa\in\dR$ such that (\ref{CD}) holds  is known as the \emph{Bakry-\'Emery curvature} of the process, and its local nature allows for explicit estimates on various concrete models, both in continuous and discrete spaces \cite{MR3155209,MR4871708,MR3492631}. Its interest for us lies in the following result, which is an immediate consequence of Theorem \ref{th:main} and Proposition \ref{pr:psi}. 
	
	\begin{corollary}[Curvature and mixing window]\label{co:kappa} If \eqref{eq:ass-det-init} holds, then
		\begin{eqnarray*}
			\gamma_\star & \ge & \kappa,
		\end{eqnarray*}
		where $\kappa$ is the Bakry-\'Emery curvature. In particular, when $\kappa>0$, we have
		\begin{eqnarray}
			\label{wmixkappa}
			\forall \varepsilon\in(0,1/2),\quad \wmix(\varepsilon) & \le & \frac{3}{2\kappa}\log\frac{1}{\varepsilon}.
		\end{eqnarray}
	\end{corollary}
	A version of this curvature-based window estimate was recently established  by the second author in the very special case of diffusions \cite[Theorem 2]{salez2025diffusions}, with a worse polynomial dependence on $\frac1\varepsilon$ rather than the optimal logarithmic one of Theorem \ref{th:main}. The proof relied on a particular differential relation between varentropy and entropy, that crucially exploits the  chain rule. The latter was subsequently adapted to discrete spaces by the two authors \cite{ped-sal-2025}, but at an extra price which diverges logarithmically with the inverse  of the smallest transition probability.  In contrast, the present work does not use the chain rule at all, allowing us to get rid of this annoying  dimensional dependence. The improvement is already visible on our emblematic toy model (Example \ref{ex:cube}), where the window estimate obtained from the varentropy approach \cite{ped-sal-2025} is diverging instead of being constant. More generally, Corollary \ref{co:kappa} implies $\wmix=O(1)$ for MCMC samplers of high-temperature Ising and low-fugacity Hardcore models on bounded degree graphs,
	thanks to the curvature estimate in \cite{MR4871708}. Our approach is elementary, and unifies the discrete and continuous settings. Before presenting it, we describe a few easy refinements of the above results. 
	
	\begin{remark}[Beyond positive curvature] While sufficient to yield cutoff in many interesting models, the estimate  $\gamma_\star\ge \kappa$ becomes useless in negative curvature. In such situations, estimating the sub-commutation function  $\psi(t)$ using its infinitesimal behavior at $t=0^+$ as in (\ref{def:kappa}) is rather pessimistic, and it might be very beneficial to look instead for relaxed versions of the $\CD(\kappa,\infty)$ criterion of the form
		\begin{eqnarray}
			\label{def:CDrelaxed}
			\forall t\ge 0,\qquad \psi(t) & \le & Ce^{-2\kappa t},
		\end{eqnarray}
		where the poor behavior of $\psi$ on short time-scales is absorbed into a prefactor $C>1$ so as to ensure a positive decay rate $\kappa>0$ at larger times. Note that by Proposition \ref{pr:psi}, this implies
		\begin{eqnarray*}
			\gamma_\star & \ge & \frac{\kappa}{C},
		\end{eqnarray*}
		which is still good enough to ensure cutoff in models where the original  curvature $-\frac12\psi'(0)$ is non-positive. This promising direction  will be explored in a subsequent work. 
	\end{remark}
	
	\begin{remark}[Restricted local Poincaré constant]A careful inspection of our proof shows that the local Poincaré constant $\gamma_\star$ appearing in the main estimate (\ref{eq:window}) can be replaced with
		\begin{eqnarray*}
			\gamma_{\varepsilon} & := & \gamma\wedge \inf_{t\in[\tmix(1-\varepsilon),\tmix(\varepsilon)]}\gamma(\mu_t),
		\end{eqnarray*}
		which is clearly always at least as large as $\gamma_\star$, and which could potentially be much larger. We also intend to explore this in a subsequent work. 
	\end{remark}

	\begin{remark}[Worst-case initialization]When $\dX$ is finite, it is standard to consider
		\begin{eqnarray*}
			\tmix^{(\star)}(\varepsilon) & := & \max_{x\in\dX}\tmix^{(x)}(\varepsilon),
		\end{eqnarray*} 
		where the notation $\tmix^{(x)}$ indicates that the initial condition is $X_0=x$.  Since our curvature-based window estimate (\ref{wmixkappa}) is uniform in the initial state $x\in\dX$, it also implies
		\begin{eqnarray}
			\label{wmixkappa-worst}
			\forall \varepsilon\in(0,1/2),\quad \wmix^{(\star)}(\varepsilon) & \le & \frac{3}{2\kappa}\log\frac{1}{\varepsilon},
		\end{eqnarray}
		where $\wmix^{(\star)}(\varepsilon):=\tmix^{(\star)}(\varepsilon)-\tmix^{(\star)}(1-\varepsilon)$.
	\end{remark}
	\section{Proofs}
	
	Our proof of   Theorem \ref{th:main} relies on the simple and classical idea of replacing the total variation distance  by the more tractable $\chi^2$-divergence, but with the  crucial novelty that our reference measure  is \emph{not} the equilibrium law. Specifically, for any  $\mu,\nu\in\cP(\dX)$, we define 
	\begin{eqnarray*}
		\chisq{\nu}{\mu} & := & 	\Var_\mu\left(\frac{\dd \nu}{\dd \mu}\right) \ = \ \int\left(\frac{\dd \nu}{\dd \mu}-1\right)^2\,\dd \mu,
	\end{eqnarray*}
	provided $\nu\ll\mu$, and $\chisq{\nu}{\mu}:=+\infty$ otherwise. Note that, by the Cauchy-Schwarz inequality, for $\nu\ll \pi$ we have the alternative variational formulation
	\begin{eqnarray}
		\label{variational}
		\chisq{\nu}{\mu} & = &  \sup_{\substack{f\in L^2(\mu)\\ \variwr{\mu }{f}=1}} {\tond*{\int  f \dd\nu - \int f\dd\mu}^2}.
	\end{eqnarray}
	On the other hand, by the Cauchy-Schwarz inequality again, we have
	\begin{eqnarray}
		\label{CS}
		{\chisq{\nu}{\mu}} & \ge & \left(\int\left|\frac{\dd \nu}{\dd \mu}-1\right|\,\dd \mu\right)^2 \ = \ 4\dtv^2(\nu,\mu),
	\end{eqnarray}
	so in order to control  the total variation distance, it is enough to control the $\chi^2$-divergence. The latter has the key advantage of  decaying exponentially fast along the semigroup, at a rate which is exactly the generalized Poincaré constant defined at (\ref{def:LPI}). This is the content of the following lemma. The special case where $\mu=\pi$ is classical and widely used, but the non-equilibrium extension given here  turns out to be crucial for our proof. Interestingly, a version of this lemma also holds for general $\phi$-divergences \cite[Theorem 8.3.1]{che-2026-book}, but only under the additional assumption that the semigroup satisfies the diffusion chain rule. Since cutoff for this subclass was already investigated in \cite{salez2025diffusions}, and the aim of this paper is to consider arbitrary Markov semigroups, it is once again crucial for us to work with the $\chi^2$-divergence. 
	
	\begin{lemma}[Exponential decay]\label{lm:chi}Fix $\mu,\nu\in\cP(\dX)$, and suppose that $\nu\ll\mu\ll \pi$. Then,
		\begin{eqnarray*}
			\forall t\ge 0,\qquad \chisq{\nu P_t}{\mu P_t}  & \le  & \chisq{\nu}{\mu} \expo{-2\int_0^t\gamma(\mu P_s)\,\dd s}.
		\end{eqnarray*}
	\end{lemma}
	\begin{proof}
		We assume that $\chisq{\nu}{\mu} <\infty$, otherwise there is nothing to prove. By the data processing inequality \cite[Theorem 7.4]{pol-wu-2025}, this implies that for each $t\ge 0$, $\chisq{\nu P_t}{\mu P_t} <\infty$, so  that $ \frac{\dd(\nu P_t)}{\dd(\mu P_t)}$ exists and is in $L^2\tond*{\mu P_t}$. Note that we also have $\mu P_t \ll \pi$: indeed, if $A\in\cE$ satisfies $\pi(A)=0$, then by stationarity we have $P_t 1_A = 0$ $\pi-$a.e., hence  $\mu-$a.e. (as $\mu\ll\pi$), so that
		\begin{eqnarray*}
			(\mu P_t)(A) & = & \int P_t{1_A}\,\dd \mu
			\ = \ 0.
		\end{eqnarray*}		
		Our starting point is the classical integral identity 
		\begin{eqnarray}\label{eq:integral-identity-gamma}
			P_t f^2 -\tond*{P_{t} f}^2 & = & 2 \int_0^t P_s \Gamma P_{t-s} f \dd s,
		\end{eqnarray}
		valid for any $f\in \cA$ and any $t\ge 0$ (see, e.g., \cite{MR3155209}).  We now fix $f\in\cA$ and $0\le t\le T$.  Replacing $f$ by $P_{T-t} f \in \cA$ in  \eqref{eq:integral-identity-gamma} and integrating against $\mu$, we obtain
		\begin{eqnarray*}
			\variwr{\mu P_t}{P_{T-t}f} - \variwr{\mu}{P_T f} & = & \int_{\dX}  P_t \tond*{P_{T-t}f}^2\, \dd \mu - \int_{\dX}  (P_T f)^2\, \dd\mu 
			\\
			& = & 2 \int_0^t \int_{\dX}  P_s \Gamma (P_{T-s}f)\, \dd\mu \,\dd s
			\\
			& =  & 2 \int_0^t \int_{\dX} \Gamma (P_{T-s}f) \dd (\mu P_s) \,\dd s
			\\
			& \ge & 2 \int_0^t \gamma(\mu P_s) \variwr{\mu P_s}{P_{T-s} f}\, \dd s,
		\end{eqnarray*}
		where the last line uses the definition of $\gamma(\mu P_s)$. An application of Gr\"onwall's inequality to the function $t\to \variwr{\mu P_t}{P_{T-t}f}$ leads to the conclusion
		\begin{eqnarray}
			\label{eq:1}
			\variwr{\mu}{P_T f} & \le & \expo{-2\int_0^T \gamma(\mu P_s) \dd s}
			\variwr{\mu P_T}{f}.
		\end{eqnarray}
		Now, we observe that $\cA$ is dense in $L^2(\mu P_T)$: this follows from an approximation argument, which uses the assumptions on the algebra $\cA$
		together with the fact that $\mu P_T \ll \pi$. Thus, \eqref{eq:1} extends to all $f\in L^2(\mu P_T)$. In particular, it applies to the function $f := \relden{(\nu P_T)}{(\mu P_T)}\in L^2(\mu P_T)$. But the latter satisfies the double identity
		\begin{eqnarray}
			\label{eq:2}
			\chisq{\nu P_T}{\mu P_T}  & = & \variwr{\mu P_T}{f} \ = \ \int P_T f \dd\nu - \int P_T f\dd\mu,
		\end{eqnarray}
		while the Cauchy--Schwarz inequality ensures that
		\begin{eqnarray}
			\label{eq:3}
			\tond*{\int P_T f\dd\nu - \int P_T f\dd\mu}^2 & =  & \tond*{\int (F_0 -1) P_T f \dd \mu }^2 
			\ \le \ \chisq{\nu}{\mu}\variwr{\mu}{P_Tf}.
		\end{eqnarray}
		Combining (\ref{eq:1}), (\ref{eq:2}) and (\ref{eq:3}) readily yields the desired inequality. 
	\end{proof}
	
	We will also need the following lemma, which is borrowed from \cite{pedrotti2025transportapproachcutoffphenomenon}.
	\begin{lemma}[Interpolation]\label{lm:interp} Fix two measures $\mu,\pi\in\cP(\dX)$ with $\dtv(\mu,\pi)<1$. Then, there is a   measure $\nu\in\cP(\dX)$ which is absolutely continuous w.r.t. both $\mu$ and $\pi$, such that
		\begin{eqnarray*}
			\left\|\frac{\dd \nu}{\dd \mu}\right\|_\infty\vee\ \ \left\|\frac{\dd\nu}{\dd\pi}\right\|_\infty	 & \leq & \frac{1}{1-\dtv(\mu,\pi)}.
		\end{eqnarray*}
	\end{lemma}
	\begin{proof}
		Fix a reference measure $\sigma$ with respect to which both  $\mu$ and $\pi$ are absolutely continuous (e.g., $\sigma:=\mu+\pi$), and let $f:=\frac{\dd\mu}{\dd\sigma}$ and $g:=\frac{\dd\pi}{\dd\sigma}$ denote the corresponding Radon-Nikodym derivatives. With this notation at hand, we  have the integral representation
		\begin{eqnarray*}
			\dtv(\mu,\pi) & = & 1-\int {f \wedge g}\, \dd\sigma.
		\end{eqnarray*}
		Consequently, we can define a probability measure $\nu$ by the formula
		\begin{eqnarray*}
			\dd \nu & := & \frac{f\wedge g}{1-\dtv(\mu,\pi)}\,\dd{\sigma}.
		\end{eqnarray*}
		This measure is absolutely continuous w.r.t. both $\mu$ and $\pi$, because $\nu(A)\le \frac{\mu(A)\wedge \pi(A)}{1-\dtv({\mu},{\pi})},$ for every $A\in\cE$. Moreover, the corresponding Radon-Nikodym derivatives are
		\begin{eqnarray*}
			\frac{\dd \nu}{\dd \mu} \ = \ \frac{1\wedge \frac g f}{1-\dtv(\mu,\pi)},& \textrm{ and } & 
			\frac{\dd\nu}{\dd \pi} \ = \ \frac{1\wedge \frac f g}{1-\dtv(\mu,\pi)},
		\end{eqnarray*}
		those formulae being interpreted as zero outside  $\supp(\nu):=\{f\wedge g>0\}$. 
	\end{proof}
	We now have all we need to prove our main result.
	\begin{proof}[Proof of Theorem \ref{th:main}]
		Fix $\varepsilon\in(0,1/2)$ and set $t_0:=\tmix(1-\varepsilon)+h$ for a small $h>0$.
		By Lemma \ref{lm:interp}, there exists $\nu\in\cP(\dX)$ whose densities w.r.t. $\pi$ and $\mu_{t_0}$ are both bounded by $1/\varepsilon$. In particular, 
		\begin{eqnarray*}
			\chisq{\nu}{\pi} \vee  \chisq{\nu}{\mu_{t_0}} & \le & \frac{1}{{\varepsilon}}.
		\end{eqnarray*}
		In view of  Lemma \ref{lm:chi} and our definition of $\gamma_\star$, we deduce that  for all $t\ge 0$,
		\begin{eqnarray*}
			\chisq{\nu P_t}{\pi} \vee \chisq{\nu P_t}{\mu_{t_0} P_t}  & \le & \frac{ e^{-2\gamma_\star t}}{{\varepsilon}}.
		\end{eqnarray*}
		Using  the triangle inequality and (\ref{CS}), we arrive at
		\begin{eqnarray*}
			\dtv\tond*{\mu_0 P_{t_0+t},\pi} & = & \dtv(\mu_{t_0} P_t,\pi) \\
			& \le &  \dtv(\nu P_t,\pi) + \dtv(\nu P_t,\mu_{t_0} P_t)\\
			& \le & \frac{1}{2}
			\tond*{\sqrt{\chisq{\nu P_t}{\pi}}+\sqrt{\chisq{\nu P_t}{\mu_{t_0} P_t}} }
			\\
			& \le &  \frac{ e^{-\gamma_\star t}}{\sqrt{\varepsilon}}.
		\end{eqnarray*}
		Finally, choosing $t:=\frac{3}{2\gamma_\star}\log \frac{1}{\varepsilon}$ ensures that $\dtv\tond*{\mu_0 P_{t_0+t},\pi}\le\varepsilon$, hence $\tmix(\varepsilon)\le t_0+t$. Recalling that $t_0=\tmix(1-\varepsilon)+h$, and that $h>0$ was arbitrary, this implies the claimed inequality (\ref{eq:window}) by letting $h\to 0$.
	\end{proof}
	Finally, let us show how the sub-commutation property (\ref{def:psi}) leads to an estimate on $\gamma_\star$.  
	\begin{proof}[Proof of Proposition \ref{pr:psi}]
		Fix $f\in \cA$ and $t>0$.
		
		By \eqref{eq:integral-identity-gamma}, the definition of $\psi(s)$ and the semigroup property, we have that
		\begin{eqnarray*}
			(P_t f^2)-(P_tf)^2 & = & 2 \int_0^t P_s\Gamma P_{t-s}f\,\dd s.
			\\
			& \le & \tond*{2 \int_0^t \psi(s)  \,\dd s} P_t \Gamma f.
		\end{eqnarray*}
		If $\mu_0\ll \pi$, then  we can  safely integrate this with respect to $\mu_0$ to deduce that 
		\begin{eqnarray} \label{eq:expectation-variance}
			\int \tond*{P_t f^2 - (P_t f)^2} \,\dd\mu_0   & \le & \tond*{2 \int_0^t \psi(s)  \,\dd s} \int \Gamma f \, \dd \mu_t.
		\end{eqnarray}
		By  the law of total variance, we then obtain
		\begin{eqnarray*}
			\variwr{\mu_t}{f} & = & \int \tond*{P_t f^2 - (P_t f)^2} \,\dd\mu_0 + \variwr{\mu_0}{P_t f}
			\\
			& \le & 
			\tond*{2 \int_0^t \psi(s)  \,\dd s} \int \Gamma f \, \dd \mu_t + \frac1{\gamma(\mu_0)}\int \Gamma(P_t f) \, \dd\mu_0.
			\\
			& \le &  \tond*{2 \int_0^t \psi(s)  \,\dd s} \int \Gamma f \, \dd \mu_t + \frac{\psi(t)}{\gamma(\mu_0)} \int \Gamma(f) \, \dd\mu_t,
		\end{eqnarray*}
		where we have used the definition of $\gamma(\mu_0)$ and  $\psi(t)$.
		This proves the first claim. Now, let us assume \eqref{eq:ass-det-init} instead of $\mu_0 \ll \pi$. For $0<s<t$, we can then write
		\begin{eqnarray*}
			\variwr{\mu_t}{f} 
			& = &  \variwr{\mu_s}{P_{t-s}f} +\int \tond*{P_{t-s}f^2-(P_{t-s}f)^2} \,\dd\mu_s\\
			& \le &   2\variwr{\mu_s}{P_{t}f}+2\variwr{\mu_s}{P_{t-s}f-P_t f}+\int \tond*{P_{t-s}f^2-(P_{t-s}f)^2} \,\dd\mu_s\\
			& \le &   2\variwr{\mu_s}{P_{t}f}+2\int{\left(f-P_sf\right)^2}\dd \mu_t+\int \tond*{P_{t-s}f^2-(P_{t-s}f)^2} \,\dd\mu_s,
		\end{eqnarray*}
		where we have used the law of total variance, the semigroup property, and the Cauchy-Schwarz inequality. 
		By assumption, the first term goes to $0$ as $s\to 0$. 
		The second one tends to $0$ since $P_s f\to f \in L^2(\pi)$ as $s\to 0$, $f\in L^\infty(\pi)$ and $\mu_t \ll \pi$.
		For the third term, recalling that $\mu_s\ll \pi$, we have as in \eqref{eq:expectation-variance} that
		\begin{eqnarray*}
			\int \tond*{P_{t-s}f^2-(P_{t-s}f)^2} \,\dd\mu_s & \le  & \tond*{2 \int_0^{t-s} \psi(r)  \,\dd r} \int_{\dX} \Gamma f \, \dd \mu_{t}.
			\\
			& \le &  \tond*{2 \int_0^{t} \psi(r)  \,\dd r} \int_{\dX} \Gamma f \, \dd \mu_{t}
		\end{eqnarray*}
		and the second claim follows immediately.

	\end{proof}

	\section*{Acknowledgment}
	This work is supported by the ERC consolidator grant CUTOFF (101123174).

	\bibliographystyle{plain}
	\bibliography{references}

\begin{thebibliography}{10}

\bibitem{aldous1983mixing}
David Aldous.
\newblock Random walks on finite groups and rapidly mixing {M}arkov chains.
\newblock In {\em Seminar on probability, {XVII}}, volume 986 of {\em Lecture Notes in Math.}, pages 243--297. Springer, Berlin, 1983.

\bibitem{aldous1986shuffling}
David Aldous and Persi Diaconis.
\newblock {Shuffling cards and stopping times}.
\newblock {\em American Mathematical Monthly}, pages 333--348, 1986.

\bibitem{MR889476}
D.~Bakry and Michel \'{E}mery.
\newblock Diffusions hypercontractives.
\newblock In {\em S\'{e}minaire de probabilit\'{e}s, {XIX}, 1983/84}, volume 1123 of {\em Lecture Notes in Math.}, pages 177--206. Springer, Berlin, 1985.

\bibitem{MR3155209}
Dominique Bakry, Ivan Gentil, and Michel Ledoux.
\newblock {\em Analysis and geometry of {M}arkov diffusion operators}, volume 348 of {\em Grundlehren der Mathematischen Wissenschaften [Fundamental Principles of Mathematical Sciences]}.
\newblock Springer, Cham, 2014.

\bibitem{MR3650406}
Riddhipratim Basu, Jonathan Hermon, and Yuval Peres.
\newblock Characterization of cutoff for reversible {M}arkov chains.
\newblock {\em Ann. Probab.}, 45(3):1448--1487, 2017.

\bibitem{MR3185193}
St\'ephane Boucheron, G\'abor Lugosi, and Pascal Massart.
\newblock {\em Concentration inequalities}.
\newblock Oxford University Press, Oxford, 2013.
\newblock A nonasymptotic theory of independence, With a foreword by Michel Ledoux.

\bibitem{che-2026-book}
Sinho Chewi.
\newblock {\em Log-concave sampling}.
\newblock Forthcoming, 2026.
\newblock Available online at \url{https://chewisinho.github.io/}.

\bibitem{diaconis1996cutoff}
Persi Diaconis.
\newblock The cutoff phenomenon in finite {M}arkov chains.
\newblock {\em Proc. Nat. Acad. Sci. U.S.A.}, 93(4):1659--1664, 1996.

\bibitem{diaconis1981generating}
Persi Diaconis and Mehrdad Shahshahani.
\newblock {Generating a random permutation with random transpositions}.
\newblock {\em Probability Theory and Related Fields}, 57(2):159--179, 1981.

\bibitem{ding2010total}
Jian Ding, Eyal Lubetzky, and Yuval Peres.
\newblock Total variation cutoff in birth-and-death chains.
\newblock {\em Probab. Theory Related Fields}, 146(1-2):61--85, 2010.

\bibitem{MR3492631}
Bo'az Klartag, Gady Kozma, Peter Ralli, and Prasad Tetali.
\newblock Discrete curvature and abelian groups.
\newblock {\em Canad. J. Math.}, 68(3):655--674, 2016.

\bibitem{MR2774096}
Eyal Lubetzky and Allan Sly.
\newblock Explicit expanders with cutoff phenomena.
\newblock {\em Electron. J. Probab.}, 16:no. 15, 419--435, 2011.

\bibitem{MR4871708}
Francesco Pedrotti.
\newblock Contractive coupling rates and curvature lower bounds for {M}arkov chains.
\newblock {\em Ann. Appl. Probab.}, 35(1):196--250, 2025.

\bibitem{ped-sal-2025}
Francesco Pedrotti and Justin Salez.
\newblock A new cutoff criterion for non-negatively curved chains, 2025.

\bibitem{pedrotti2025transportapproachcutoffphenomenon}
Francesco Pedrotti and Justin Salez.
\newblock A transport approach to the cutoff phenomenon, 2025.

\bibitem{pol-wu-2025}
Yury Polyanskiy and Yihong Wu.
\newblock {\em Information theory. {From} coding to learning}.
\newblock Cambridge: Cambridge University Press, 2025.

\bibitem{MR4546624}
Justin Salez.
\newblock Universality of cutoff for exclusion with reservoirs.
\newblock {\em Ann. Probab.}, 51(2):478--494, 2023.

\bibitem{MR4780485}
Justin Salez.
\newblock Cutoff for non-negatively curved {M}arkov chains.
\newblock {\em J. Eur. Math. Soc. (JEMS)}, 26(11):4375--4392, 2024.

\bibitem{salez2025diffusions}
Justin Salez.
\newblock Cutoff for non-negatively curved diffusions.
\newblock {\em Inventiones Mathematicae}, 2025.

\bibitem{salez2025lecture}
Justin Salez.
\newblock Modern aspects of markov chains: entropy, curvature and the cutoff phenomenon, 2025.

\end{thebibliography}

\end{document}